\documentclass[12pt,a4paper]{article}

\usepackage[whole]{bxcjkjatype}

\usepackage{amsmath}
\usepackage{amssymb}
\usepackage{bussproofs}
\usepackage{tikz}
\usepackage{stmaryrd}
\usepackage{cmll}
\usepackage{mathptmx}
\usepackage{xcolor}
\usepackage{amsthm}
\usepackage{framed}
\usepackage{here}
\usepackage{comment}
\usepackage{seqsplit}
\usepackage{url}

\newcommand{\tensor}{\mbox{$\otimes$}}
\newcommand{\pa}{\bindnasrepma}
\newcommand*{\oldneg}{\mathord{\sim}}

\newcommand{\preq}{\preccurlyeq}

\newtheorem{defn}{Definition}

\newtheorem{thm}{Theorem}
\newtheorem{prop}{Proposition}
\newtheorem{cor}{Corollary}

\newtheorem{rem}{Remark}

\fontsize{14pt}{0pt}\selectfont
\title{Three Topics in Non-decomposability of Generalized Multiplicative Connectives}
\fontsize{12pt}{0pt}\selectfont
\author{Yuki Nishimuta}
\fontsize{12pt}{0pt}\selectfont
\date{}
					
\begin{document}
\maketitle

\begin{abstract}
Danos and Regnier introduced generalized (non-binary) multiplicative connectives in Danos and Regnier \cite{DR}. They showed that there exist generalized multiplicative connectives that cannot be defined by any combination of the tensor and par rules in the multiplicative fragment of linear logic. Such connectives are called non-decomposable generalized multiplicative connectives \cite[p.192]{DR}.

The non-decomposability of logical connectives can be regarded as a proof-theoretic and syntactic counterpart of functional completeness for cut-free proofs. In this short note, we investigate Danos and Regnier's notion of non-decomposability and present three results concerning the (non-)decomposability of generalized multiplicative connectives.
\end{abstract}

\section{Introduction}   
Danos and Regnier introduced generalized (non-binary) multiplicative connectives in \cite{DR}. They showed that there exist generalized multiplicative connectives that cannot be defined by any combination of the tensor and par rules in the multiplicative fragment of linear logic satisfying the main reduction step of the cut-elimination theorem. Such connectives are called non-decomposable generalized multiplicative connectives.

The non-decomposability of logical connectives can be regarded as a proof-theoretic and syntactic counterpart of functional completeness for cut-free proofs. Danos and Regnier defined generalized multiplicative connectives only in the classical multiplicative fragment of  linear logic.

In this short note, we investigate the (non-)decomposability of generalized multiplicative connectives in other fragments of linear logic and present three results on (non-)decomposability.

In the first part of this note, we define intuitionistic generalized multiplicative connectives by employing polarities, inspired by \cite{Andreoli}. We show that all intuitionistic generalized multiplicative connectives in intuitionistic multiplicative linear logic $\mathsf{IMLL}$ are decomposable. Hence, there are no non-decomposable multiplicative connectives in $\mathsf{IMLL}$.

Secondly, we generalize Danos and Regnier’s notion of (non-)decomposability and show that all non-decomposable generalized multiplicative connectives become decomposable in multiplicative additive linear logic $\mathsf{MALL}$ and multiplicative exponential linear logic $\mathsf{MELL}$.

Finally, we show that elementary multiplicative linear logic $\mathsf{EMLL}$ preserves the non-decomposability of generalized multiplicative connectives.

%%%%%%%%%%%%%%%%%%%%%%%%%%%%%%%%%%%%%%%%%%%%%%%%%%%%%%%%%

  \section{Decomposability of intuitionistic generalized connectives} 
In this section, we define intuitionistic generalized connectives by modifying the definition of (classical) generalized multiplicative connectives. The main result of this section is that all intuitionistic generalized connectives in $\mathsf{IMLL}$ are decomposable. This implies that there are no non-decomposable generalized connectives in $\mathsf{IMLL}$. We assume familiarity with generalized multiplicative connectives (see \cite{DR, Nishimuta1}).

In what follows, we define intuitionistic generalized connectives.

\medskip

First, we define polarized partitions. Then, we introduce intuitionistic polarized partitions and intuitionistic polarized meeting graphs.

\begin{defn}\label{polarized partition}
We call an occurrence of an index $s$ positive if it is unchecked, and negative if it is checked. A partition of the set $\{1,\dots,n\}$ is a polarized partition if each occurrence of an index carries one of these polarities.

\end{defn}

In $\mathsf{IMLL}$, a partition set $P_{\mathcal{C}}$ of a generalized connective $\mathcal{C}$ corresponds to  one-sided sequents.  A formula $A_n$ appearing on the right-hand side of a sequent corresponds to the checked number $\check{n}$. Hence, a sequent of $\mathsf{IMLL}$ $A_1,\dots,A_{n-1}, \oldneg A_n\vdash$ corresponds to the partition  $\{1,\dots, n-1, \check{n}\}$. When an atomic formula (a metavariable) $A$ has positive (resp. negative) polarity, we often denote $A$ as $A^{+}$ (resp. $A^{-}$).

\begin{defn}\label{intui partition}
A polarized partition $p$ is an intuitionistic polarized partition if each class of $p$ contains at most one checked element.

\end{defn}

\medskip

\par An intuitionistic polarized meeting graph $\mathcal{G}(p,q)$ is a polarized meeting graph obtained from two intuitionistic polarized partitions $p$ and $q$ as follows: we draw an edge between two nodes $N\in p$ and $N'\in q$ if there exists an index $s$ such that $s\in N$ and $\check{s}\in N'$ (or $\check{s}\in N$ and $s\in N'$). Thus each index $s$ induces at most one edge, namely between the unique classes containing $s$ and $\check{s}$.

% cite our paper

%\medskip

%\par We  consider the set $\mathcal{F}$ of polarized partitions of $\{1,\dots, n \}$. We assume that a partition $p$ has the form $p=\{(p_{11},\dots,p_{1j}),\dots(p_{k1},%\dots,p_{km_{j}})\}$, $p_{ij}\in\{1,\dots, n \}$, $p\in\mathcal{F}$.

%\begin{defn}\label{intui meeting graph}
%For any polarized partitions $p,q\in\mathcal{F}$, we define a polarized meeting graph  $\mathcal{G}(p,q)=(V_{1},V_{2},E)$ as follows;
%\begin{itemize}
%\item $V_{1}$ is the set of upper nodes:  these nodes have labels such that the numbers in the same class of $p$ correspond to the same node.
%\item $V_{2}$ is the set of lower nodes:   these nodes have labels such that the numbers in the same class of $q$ correspond to the same node.
%\item$E_1$ is the set of edges connecting a node of $V_{1}$ and a node of $V_{2}$: if two numbers $n\in N$ and $n'\in N'$ are the same number and one of them is checked and the other is not  (where $N$ ($N'$) is $V_{1}$ ($V_{2}$) or $V_{2}$ ($V_{1}$), respectively). 

%\end{itemize}
%\end{defn}

\begin{defn}
An  $n$-ary intuitionistic generalized connective $\mathcal{C}$ is a pair of finite sets of intuitionistic polarized partitions of $\{1,\dots,n\}$, $(P_L,P_R)$ such that $(P_R)^{\bot}=P_L$ and $(P_L)^{\bot}=P_R$ hold.
\end{defn}

The definition of (non-)decomposability for intuitionistic generalized multiplicative connectives is analogous to that for (classical) generalized multiplicative connectives.

\bigskip

The general form of a right introduction rule for (two-sided) intuitionistic generalized multiplicative connectives is as follows:

\begin{center}
\def\fCenter{\ \vdash\ }
\Axiom$\Gamma_1, A_{11},\dots, A_{1m_{1}-1}\fCenter A_{1m_{1}} \ \ \cdots \ \ \Gamma_k, A_{k1},\dots, A_{km_{k}-1}\fCenter A_{km_{k}}$
\UnaryInf$\Gamma \fCenter\mathcal{C}(A_{11},\dots,A_{km_k}) $
\DisplayProof

\end{center}
where $\Gamma=\bigcup^{k}_{i=1}\Gamma_i$ and $j_{m_j}\in\{1,\dots,n\}$.

\medskip

We now show that all intuitionistic generalized multiplicative connectives are decomposable.

%orthogonalityは説明するか
\begin{thm}\label{IMLL decomp}
Let $\mathcal{C}$ be an arbitrary intuitionistic multiplicative generalized connective. Then, $\mathcal{C}$ is $\mathsf{IMLL}$-decomposable .
\end{thm}
 
 \begin{proof}
  Let $L_1, \dots,L_k$ be an enumeration of left introduction rules of  $\mathcal{C}$ and $R_1,\dots,$
  \noindent$R_m$ be that of right introduction rules. The main reduction step of the \seqsplit{cut-elimination} for a $\mathcal{C}$-connective holds if and only if the two intuitionistic partition sets of $\mathcal{C}$, $(P_L, P_R)$ are orthogonal. This equivalence is proved in \cite[Lemma 2]{DR}. By definition of $\mathcal{C}$,  $\mathcal{C}$ satisfies the main reduction step of the cut-elimination. We prove this theorem by case analysis on the number and the form of right introduction rules.
% case 1
\par\noindent Case $\#P_R=1$: $\mathcal{C}$ is decomposed as follows. For each class $x_i$ ($i=1,\dots, k$), we apply the left $\tensor$-rules as much as possible. We use $\multimap$ and obtain a formula $\alpha_i$. We connect these formulas using the tensor rules $\alpha=\alpha_1\tensor\dots\tensor\alpha_k$. In the following, we call this form of formulae normal. By construction, $P_{\mathcal{C}}=P_{\alpha}$ holds. Hence, $\mathcal{C}$ is decomposable.
%case 2
\par\noindent Case $\#P_R>1$: Consider two arbitrary right rules of $\mathcal{C}$ (say, $R_1, R_2\in\{R_1, \dots, $
\noindent $R_k\}$).

\par\noindent Subcase (i): For some formula $A$, the polarity of $A$ changes  between $R_1$ and $R_2$ (e.g. $A^{+}\in R_1$ and $A^{-}\in R_2$). In this case, cut elimination fails. For example, let $R_1=\{(A, B \vdash E) \ \  (C, D\vdash F)\}$ and $R_2=\{(F, B\vdash E) \ \ (C, D\vdash A)\}$. The dual of $R_1$ should be $L_1=\{(F\vdash) (\vdash A)\cdots\}$, where we omitted  irrelevant formulas. The meeting graph between $R_2$ and $L_1$ is cyclic  and the cut-elimination does not  hold. 
%case 3
\par\noindent Subcase (ii): For all formulas $A_i$,  a polarity of $A_i$ does not change in $R_k$. In this case,  if we apply the right $\tensor$-rules as much as possible, we obtain one sequent. After that, we apply the left $\tensor$-rules as much as possible and then apply one right $\multimap$-rule. By this construction, we can obtain the same formula $\alpha$ from $R_k$ for each $k$.  We show that $\alpha$ is the decomposition of $\mathcal{C}$. We assume that $\#P_{\alpha}>\#P_R$ and show a contradiction. Consider arbitrary two elements $p\in(P_{\alpha}\setminus P_R)$ and $q\in P_{\oldneg\alpha}$. By the relation $P_{\alpha}\supset P_R$, $P_L\supset P_{\oldneg\alpha}$ holds. Hence $q\in P_L$ holds. It contradicts  the definition of an intuitionistic generalized connective $(P_L)^{\bot}=P_R$ because $p\bot q$ and $p\notin P_R$ hold. Therefore,  $\#P_{\alpha}=\#P_R$ holds.
\end{proof}

\begin{cor}
There are no non-decomposable multiplicative connectives in $\mathsf{IMLL}$.
\end{cor}

%general eliminationについてのリマーク

%%%%%%%%%%%%%%%%%%%%%%%%%%%%%%%%%%%%%%%%%%%%%%%%%%%%%%%%%

% section名変更の可能性あり
 \section{Decomposability of generalized multiplicative connectives in $\mathsf{MALL}$ and $\mathsf{MELL}$}
 
 Danos and Regnier defined the notion of (non-)decomposability of generalized connectives only in the \emph{multiplicative} framework \cite{DR}.  

In order to investigate whether non-decomposability is preserved in other logical systems, we extend Danos and Regnier’s definition of non-decomposability. We then show that multiplicative additive linear logic $\mathsf{MALL}$ and multiplicative exponential linear logic $\mathsf{MELL}$ do \emph{not} preserve non-decomposability.

  \bigskip

{\bf Notation and terminology}; The letter $\mathbb{L}$ represents some logical system containing $\mathsf{MALL}$ as a subsystem.  We call the formulas of the form  $(A_{11}\pa\dots\pa A_{1i_1})$

\noindent $\tensor\dots\tensor(A_{m1}\pa\dots\pa A_{mi_m})$ normal (where $m \ge 1$ and $i_j \ge 0$).
% index

\bigskip

We define the order relation $\preq_{\mathbb{L}}$ as follows; for any formulas $P, Q\in\mathbb{L}$, $P\preq_{\mathbb{L}} Q$ if $Q$ is derivable from $P$ in the one-sided sequent calculus $\mathbb{L}$ using the rules of  $\mathbb{L}$.  We  omit the subscript $\mathbb{L}$ for readability.
\begin{defn}\label{L-decomp}
Let $I_1,\dots,I_k$ ($k\in\mathbb{N}$) be the introduction rules of a generalized
multiplicative connective $\mathcal{C}$, and let $S_{i1},\dots,S_{im_i}$ ($i=1,\dots,k$)
be the premises of $I_i$.
A generalized connective $\mathcal{C}$ is said to be \emph{$\mathbb{L}$-decomposable}
if one of the following holds:
\begin{enumerate}
\item there exists a formula $\alpha\in\mathsf{MLL}$ such that
      $P_{\mathcal{C}}=P_{\alpha}$ holds using only inference rules of $\mathsf{MLL}$;
\item there exists a formula $\alpha\in\mathbb{L}$ such that
  \begin{enumerate}
  \item[(i)] for all $i$ ($1\le i\le k$), $\vdash \alpha$ is derivable from
        $S_{i1},\dots,S_{im_i}$ using at least one inference rule of $\mathbb{L}$
        not belonging to $\mathsf{MLL}$;
  \item[(ii)] each $S_{i1},\dots,S_{im_i}$ is obtainable from $\vdash \alpha$
        by bottom-up construction using at least one inference rule of $\mathbb{L}$
        not belonging to $\mathsf{MLL}$;
  \item[(iii)] $\alpha$ is minimal with respect to $\preq$ among all formulas
        satisfying (i) and (ii), that is, for any $\beta$ satisfying (i) and (ii),
        we have $\alpha \preq \beta$.
  \end{enumerate}
\end{enumerate}
\end{defn}

\begin{rem}
The minimality condition (iii) excludes trivial decompositions obtained by adding superfluous connectives and ensures that decomposability reflects a genuine increase in expressive power of the logical system.
For instance, if $\alpha$ satisfies Definition~\ref{L-decomp}(2), then the formula $\alpha \oplus \alpha$ also satisfies conditions (i) and (ii), but it carries no additional information about the partition set; the minimality requirement rules out such redundant decompositions.
\end{rem}

We show that all of generalized  multiplicative connectives are decomposable in $\mathsf{MALL}$ and $\mathsf{MELL}$.

%additive disjunctionによってCはdecomposableになる
\begin{prop}\label{additive decomp}
Let $\mathcal{C}$ be an arbitrary $\mathsf{MLL}$-non-decomposable generalized connective. $\mathcal{C}$ is $\mathsf{MALL}$-decomposable.
\end{prop}

\begin{proof}
Let $P_{\mathcal{C}}=\{p_1,\dots,p_s\}$ ($s\in\mathbb{N}$) be the partition set of  $\mathcal{C}(A_1,\dots,A_m)$. For each partition $p_j$, the unique (up to commutativity and associativity) normal formula $\alpha_j$ ($j\in\{1,\dots, s\}$) having the same assumptions as $p_j$ is uniquely determined. Put $\alpha=\alpha_1\oplus\dots\oplus\alpha_s$. $\vdash \alpha_1\oplus\dots\oplus\alpha_s$ is derivable from $p_j$ ($j=1,\dots, s$). For each $j$, $p_j$ is obtainable by bottom-up construction. %Obviously,  $\alpha$ is minimal with respect to $\preq$.
Moreover, $\alpha$ is minimal with respect to $\preq$ among all formulas satisfying Definition~\ref{L-decomp}(2). Indeed, for each $j$ the premises corresponding to $p_j$ derive $\vdash \alpha_j$, and hence any $\beta$ satisfying (i) must satisfy $\alpha_j \preq \beta$ for all $j$. Therefore, by the $\oplus$-rule, $\alpha_1\oplus\cdots\oplus\alpha_s \preq \beta$.
\end{proof}

Figure~\ref{MALL decomposition} is an example of decomposition for the Danos and Regnier's non-decomposable connective in $\mathsf{MALL}$.

\begin{figure}[h]
\begin{framed}
\begin{center}

\medskip

\def\fCenter{\ \vdash\ }
\Axiom$\fCenter A, B, \Gamma$
\UnaryInf$\fCenter A\pa B, \Gamma$
\Axiom$\fCenter C, D, \Delta$
\UnaryInf$\fCenter C\pa D, \Delta$
\BinaryInf$ \fCenter(A\pa B)\tensor(C\pa D), \Gamma, \Delta$
\UnaryInf$\fCenter ((A\pa B)\tensor(C\pa D))\oplus ((A\pa C)\tensor(B\pa D)), \Gamma, \Delta$
\DisplayProof

\bigskip

\def\fCenter{\ \vdash\ }
\Axiom$\fCenter A, C, \Gamma$
\UnaryInf$\fCenter A\pa C, \Gamma$
\Axiom$\fCenter B, D, \Delta$
\UnaryInf$\fCenter B\pa D, \Delta$
\BinaryInf$ \fCenter(A\pa C)\tensor(B\pa D), \Gamma, \Delta$
\UnaryInf$\fCenter ((A\pa B)\tensor(C\pa D))\oplus((A\pa C)\tensor(B\pa D)), \Gamma, \Delta$
\DisplayProof

\end{center}
\caption{Decomposition of a generalized connective in $\mathsf{MALL}$}\label{MALL decomposition}
\end{framed}
\end{figure}

\begin{rem}
When we restrict attention to cut-free proofs, a generalized connective $\mathcal{C}(A_1, \dots, A_m)$ can be identified with the formula $\alpha_1 \oplus \dots \oplus \alpha_s$ constructed above.  
However, this identification fails in the presence of cuts. Indeed, the dual connective $\mathcal{C}^{\ast}(\oldneg A_1, \dots, \oldneg A_m)$ can be decomposed as $\beta_1 \oplus \dots \oplus \beta_s$ by the same method, but the formulas $\alpha_1 \oplus \dots \oplus \alpha_s$ and $\beta_1 \oplus \dots \oplus \beta_s$ are not De Morgan duals. Consequently, the cut rule is not defined between these formulas (see Figure~\ref{fig:undefined-cut}).
\end{rem}

%例を詳しく

If we consider only cut-free proofs, the introduction of non-decomposable connectives amounts to the introduction of additive disjunction in the restricted form. 

\begin{figure}[h]
\begin{framed}
\begin{center}
\scalebox{0.6}{
\def\fCenter{\ \vdash\ }
\Axiom$\fCenter\Gamma_1, A, B$
\Axiom$\fCenter\Gamma_2, C, D$
\doubleLine
\BinaryInf$ \fCenter\Gamma_1, \Gamma_2, (A\pa B)\tensor(C \pa D)\oplus(A\pa C)\tensor (B\pa D)$
\Axiom$\fCenter\Delta_1, \oldneg A, \oldneg D$
\Axiom$\fCenter\Delta_2, \oldneg B$
\Axiom$\fCenter\Delta_3, \oldneg C$
\doubleLine
\TrinaryInf$ \fCenter\Delta_1,\Delta_2,\Delta_3,((\oldneg A\pa\oldneg D)\tensor\oldneg B\tensor\oldneg C)\oplus ((\oldneg B\pa\oldneg C)\tensor \oldneg A\tensor\oldneg D)$
\BinaryInf$ \fCenter\Gamma_1, \Gamma_2, \Delta_1,\Delta_2,\Delta_3$
\DisplayProof}
\end{center}

\begin{center}
\scalebox{0.7}{
\def\fCenter{\ \vdash\ }
\Axiom$\fCenter\Gamma_1, A, B$
\Axiom$\fCenter\Delta_1, \oldneg A,\oldneg D$
\BinaryInf$ \fCenter\Gamma_1,  \Delta_1, B, \oldneg D$
\Axiom$\fCenter\Delta_2, \oldneg B$
\BinaryInf$ \fCenter\Gamma_1,  \Delta_1, \Delta_2, \oldneg D$
\Axiom$\fCenter\Gamma_2, C, D$
\BinaryInf$ \fCenter\Gamma_1, \Gamma_2, \Delta_1,\Delta_2, C$
\Axiom$\fCenter\Delta_3, \oldneg C$
\BinaryInf$ \fCenter\Gamma_1, \Gamma_2, \Delta_1,\Delta_2, \Delta_3$
\DisplayProof}
\end{center}
\caption{Undefined cut between $\oplus$ and $\oplus$}
\label{fig:undefined-cut}
\end{framed}
\end{figure}

%モデルとdeducibility of identicalsについてのリマーク

%%%%%%%%%%%%%%%%%%%%%%%%%%%%%%%%%%%%%%%%%

\section{Non-decomposability preservation result}

Let $\mathsf{EMLL}$ denote the multiplicative-exponential fragment of Elementary Linear Logic. In this section, we show that  $\mathsf{EMLL}$ preserves non-decomposability. Namely, if a generalized multiplicative connective $\mathcal{C}$ is non-decomposable in $\mathsf{MLL}$, then $\mathcal{C}$ remains non-decomposable in $\mathsf{EMLL}$.

The multiplicative-exponential fragment of Elementary Linear Logic $\mathsf{EMLL}$ is defined as follows.
\begin{defn}
The formulas of $\mathsf{EMLL}$ are given by:

$$A:= P | \oldneg P | A\tensor A | A\pa A|!A|?A$$

Negation is inductively defined  as follows;

$$\oldneg\oldneg P:=P, \oldneg(A\tensor B):=\oldneg A\pa\oldneg B, \oldneg(A\pa B):=\oldneg A\tensor\oldneg B,$$ $$ \oldneg(!A):=?\oldneg A, \oldneg(?A)=!\oldneg A$$
(where $P$ is an atomic formula）

\end{defn}

The inference rules of $\mathsf{EMLL}$ are given in Figure \ref{EMLL rules}.

\begin{figure}[h]
\begin{framed}
\begin{center}
%\RightLabel{(Ax)}
\AxiomC{}
\RightLabel{(Ax)}
\UnaryInfC{$\ \vdash\ P, \oldneg P$}
\DisplayProof
\medskip
\def\fCenter{\ \vdash\ }
\Axiom$\fCenter \Gamma,  A$
\Axiom$\fCenter\Delta, \oldneg A$
\RightLabel{(Cut)}
\BinaryInf$\fCenter\Gamma, \Delta$
\DisplayProof
\end{center}

\begin{center}
\def\fCenter{\ \vdash\ }
\Axiom$\fCenter \Gamma,  A$
\Axiom$\fCenter\Delta, B$
\RightLabel{(\tensor)}
\BinaryInf$\fCenter\Gamma, \Delta, A\tensor B$
\DisplayProof
\medskip
\Axiom$\fCenter \Gamma,  A, B$
\RightLabel{($\pa$)}
\UnaryInf$\fCenter\Gamma, A\pa B$
\DisplayProof
\medskip
\Axiom$\fCenter \Gamma$
\RightLabel{(?-weakening)}
\UnaryInf$\fCenter\Gamma, ?A$
\DisplayProof
\medskip
\Axiom$\fCenter \Gamma, ?A, ?A$
\RightLabel{(?-contraction)}
\UnaryInf$\fCenter\Gamma, ?A$
\DisplayProof
\medskip
\Axiom$\fCenter\Gamma, A$
\RightLabel{($K$)}
\UnaryInf$\fCenter?\Gamma, !A$
\DisplayProof
\end{center}
\caption{Inference rules of $\mathsf{EMLL}$ }
\label{EMLL rules}
\end{framed}
\end{figure}

The following is the general form of an introduction rule for a generalized multiplicative connective. Note that each premise of a generalized connective contains at least one principal formula.

\begin{center} 
\def\fCenter{\ \vdash\ }
\Axiom$\fCenter \Gamma_1,  A_{11}, \dots, A_{1i_1} \ \ \cdots$
\Axiom$\fCenter\Gamma_m, A_{m1},\dots,A_{mi_m}$
\BinaryInf$\fCenter\Gamma_1,\dots, \Gamma_m, \mathcal{C}(A_1,\dots,A_n)$
\DisplayProof
\end{center}
where $ji_j\in\{1,\dots,n\}$ and $i_k\neq 0 \  (1\leq k\leq m)$.

\begin{thm}\label{EMLL nondecomp}
Let $\mathcal{C}$ be an arbitrary $\mathsf{MLL}$-non-decomposable connective.  Then  $\mathcal{C}$ is $\mathsf{EMLL}$-non-decomposable.
\end{thm}

\begin{proof}

Assume, for contradiction, that $\mathcal{C}$ is $\mathsf{EMLL}$-decomposable. By definition, there is a $\mathsf{EMLL}$-formula  $\alpha$ such that $\vdash \Gamma, \alpha$ (where $\Gamma=\bigcup^{m}_{s=1}\Gamma_s$ and  ``$\cup$" is the multiset union) is derivable from the same assumptions for $\mathcal{C}$ in $\mathsf{EMLL}$. The formula $\alpha$ must contain at least one modal operator. Otherwise, it contradicts  the assumption that $\mathcal{C}$ is $\mathsf{MLL}$-non-decomposable.
The  modal operator contained in $\alpha$ cannot be a bang connective (!). In that case,  $\Gamma$ must have the form $?\Delta$. It contradicts the assumption that $\Gamma$ is arbitrary. If $\alpha$ contains a why-not connective (?), a why-not formula can be introduced only by the $K$-rule or by $?$-weakening. It cannot be the $K$-rule by the argument above. Hence, some formula $?A$ is introduced by the weakening rule. 
%
% シーケントの個数
%If $m=1$, we obtain some formula $\beta$ by removing formulas of the form $?A$ and the superfluous pa connectives from $\alpha$. Hence,  $\mathcal{C}$ is $\mathsf{MLL}$-decomposable, contradiction. Hence we assume $m>1$. It suffices to prove it in the case $m=2$.
If an operand of a tensor in the formula $\alpha$ has the form $?A$, then there is some decomposition of the formula $\alpha$  such that some sequent has only one formula $(\cdots \vdash ?A \cdots)$. We apply $?$-weakening rule to this sequent and obtain the empty sequent $\vdash$. This contradicts the definition of $\mathcal{C}$-introduction rule. Hence no occurrence of $?A$ can appear as an operand of any $\tensor$ in $\alpha$.  If any operand of every tensor connective in $\alpha$ has the form $?A$,
let $\alpha'$ be obtained from $\alpha$ by deleting all occurrences of $?A$ and the superfluous $\pa$-connectives. If $\alpha'$ still satisfied Definition~\ref{L-decomp}(2)(i)(ii), then $\alpha'$ would be a strictly simpler candidate decomposition,
contradicting the minimality of $\alpha$. Hence $\alpha'$ cannot satisfy (i)(ii). Nevertheless, by construction $\alpha'$ is still derivable from each premise of $\mathcal{C}$ using only $\mathsf{MLL}$-rules, yielding an $\mathsf{MLL}$-decomposition of $\mathcal{C}$, a contradiction.

\end{proof}

%
%\begin{prop}
%$\mathsf{ELL}(\mathcal{C}_i)$ admits the cut-elimination theorem.
%\end{prop}

%\begin{proof}
%Similar for \cite[Proposition 4.5]{Pedicini}. 
%\end{proof}

%%%%%%%%MLLLについて言葉で説明

Preservation of non-decomposability in multiplicative light linear logic $\mathsf{MLLL}$ can be proved by an argument similar to the one given above.

%有用性についてのリマーク

\section{Conclusion}

In this note we investigated Danos and Regnier’s notion of (non-)decomposability for generalized multiplicative connectives beyond the classical multiplicative setting. First, by introducing \emph{intuitionistic} generalized multiplicative connectives via polarized partitions, we showed that every such connective in $\mathsf{IMLL}$ is decomposable; hence non-decomposable generalized multiplicative connectives do not arise in the intuitionistic multiplicative fragment. Secondly, we extended the definition of decomposability to larger systems and proved that any $\mathsf{MLL}$-non-decomposable generalized connective becomes decomposable in $\mathsf{MALL}$ (and similarly in $\mathsf{MELL}$), essentially because additive (or exponential) structure allows one to assemble the set of $\mathsf{MLL}$-normal forms associated with the partition set. Finally, we showed that this collapse does not occur in $\mathsf{EMLL}$: $\mathsf{MLL}$-non-decomposable generalized connectives remain non-decomposable in the multiplicative fragment of Elementary Linear Logic. These results clarify that non-decomposability is sensitive both to the classical/intuitionistic polarity discipline and to the availability of additive or exponential resources, and they isolate $\mathsf{EMLL}$ as a natural extension of $\mathsf{MLL}$ in which non-decomposability is preserved.

%%%%%%%%%%%%%%%%%%%%%%%%%%%%%%%%%%%%%%%%%

%Schroeder


\begin{thebibliography}{20}
\bibitem{Andreoli}J.-M. Andreoli, \textit{Logic programming with focusing proofs in linear logic}, Journal of Logic and Computation, vol. 2, no. 3, 1992,  pp. 297-347.
%\bibitem{Belnap}N.D.Belnap, \textit{Tonk, Plonk and Plink}, Analysis, vol. 22, No. 6, 1962, pp. 130-134.
\bibitem{DR}V. Danos and L. Regnier, \textit{The structure of multiplicatives}, Archives for Mathematical Logic, vol. 28, 1989, pp. 181-203.
\bibitem{Gir}J.-Y. Girard, \textit{Linear Logic}, Theoretical Computer Science, vol. 50, 1987, pp. 1-102.
\bibitem{Gir 1987}J.-Y. Girard and Y. Lafont, \textit{Linear logic and lazy computation}, in: H. Ehrig , R. Kowalski , G. Levi , U. Montanari  (eds) TAPSOFT '87, Lecture Notes in Computer Science, vol 250, Springer, Berlin, 1987.
\bibitem{Girard 1991}J.-Y. Girard, \textit{A new constructive logic: classical logic}, Mathematical Structure in Computer Science, vol. 1, Issue 3, 1991, pp. 255-296.
\bibitem{Girard LLL}J.-Y. Girard, \textit{Light linear logic}, Information and Computation, vol. 143, Issue 2, 1998, pp. 175-204.
\bibitem{Maieli2}R. Maieli, \textit{Non decomposable connectives of linear logic}, Annals of Pure and Applied Logic, Vol. 170, Issue 11, 2019.
\bibitem{Nishimuta1}Y. Nishimuta and M.~Okada, \textit{A note on switching conditions for the generalized logical connectives in multiplicative linear logic},  \url{https://arxiv.org/abs/1804.02540}.
\end{thebibliography}
\end{document}